\documentclass[reqno]{amsart}
\usepackage{hyperref}
\usepackage{graphicx}
\usepackage{amssymb}
\usepackage{amsfonts}
\usepackage{amsmath}

\begin{document}
\title[\hfilneg Existence of positive solutions for a nonlinear third-order... \hfil ]
{Existence of Positive Solutions for a Nonlinear Third-order Integral Boundary Value Problem}

\author[C. Guendouz, F. Haddouchi, S. Benaicha]
{Cheikh Guendouz, Faouzi Haddouchi, Slimane Benaicha}

\address{Chiekh Guendouz \newline
 Laboratory of Fundamental and Applied Mathematics of Oran,\newline
Department of Mathematics, University of Oran 1 Ahmed Benbella,\newline
31000 Oran,
Algeria}
\email{guendouzmath@yahoo.fr}

\address{Faouzi Haddouchi \newline
Department of Physics, University of Sciences and Technology of
Oran-MB \newline
El Mnaouar, BP 1505, 31000 Oran, Algeria
\newline
Laboratory of Fundamental and Applied Mathematics of Oran,\newline
Department of Mathematics, University of Oran 1 Ahmed Benbella,\newline
31000 Oran,
Algeria}
\email{fhaddouchi@gmail.com}

\address{Slimane Benaicha \newline
Laboratory of Fundamental and Applied Mathematics of Oran,\newline
Department of Mathematics, University of Oran 1 Ahmed Benbella,\newline
31000 Oran,
Algeria}
\email{slimanebenaicha@yahoo.fr}


\begin{abstract}
 In this paper, we study the existence of at least one positive solution for a nonlinear third-order two-point boundary value problem with integral condition. By employing the Krasnoselskii's fixed point theorem on cones, the existence results of the problem are established.
\end{abstract}

\maketitle
\numberwithin{equation}{section}
\newtheorem{theorem}{Theorem}[section]
\newtheorem{lemma}[theorem]{Lemma}
\newtheorem{remark}[theorem]{Remark}
\newtheorem{definition}[theorem]{Definition}
\newtheorem{example}[theorem]{Example}
\allowdisplaybreaks

\section{Introduction}
The theory of boundary value problems is experiencing a rapid development. Many methods are used to study this kind of problems such as fixed point theorems, shooting method, iterative method with upper and lower solutions, etc. Third-order differential equation arise in a variety of different areas of applied mathematics and physics, as the deflection of a curved beam having a constant or varying cross, three layer beam, electromagnetic waves or gravity driven flows and so on \cite{Greg}.

Recently, third-order two-point or multipoint boundary value problems have attracted a lot of attention \cite{And,Agar,Bai,Feng1,Feng2,Gros,Guo,Li,Luan,Rach,Yao,Zhao}, and the references therein for related results. Among them, the fixed-point principle in cone has become an important tool used in the study of existence and multiplicity of positive solutions. Many papers that use this method have been published in recent years.\\
In this paper, we are concerned with the following third-order BVP with integral boundary condition
\begin{equation}\label{eq1}
u^{\prime \prime \prime }(t) + f(u(t)) = 0, \ t \in (0,1),
\end{equation}
\begin{equation}\label{eq2}
u(0) = u^{\prime}(0) = 0, \ u(1) = \int_{0}^{1}g(s)u(s)ds,
\end{equation}
where

\begin{itemize}
\item[(H1)] $f\in C([0,\infty),[0,\infty))$;
\item[(H2)] $g\in C([0,1],[0,\infty))$ and $ 0< \int_0 ^1 t^{2} g(t)dt<1.$
\end{itemize}

It is worth mentioning that, in 2013, Francisco. J. Torres \cite{Tor} studied the existence of positive solutions for the following nonlinear third-order three point boundary problem
\begin{equation*}
u^{\prime \prime \prime}(t) + a(t)f(t,u(t)) = 0, \ t \in (0,1),
\end{equation*}
\begin{equation*}
u(0)= 0,\ u^{\prime}(0) = u^{\prime}(1) = \alpha u(\eta).
\end{equation*}
Existence of at least one and two positive solutions are obtained by means of Krasnoselskii's fixed
point theorem and the fixed point index theory.

In \cite{Sun1}, the authors considered the following third-order boundary value problem with integral boundary conditions

\begin{equation*}
u^{\prime \prime \prime}(t) +f(t,u(t),u^{\prime}(t)) = 0, \ t \in [0,1],
\end{equation*}
\begin{equation*}
u(0) = u^{\prime}(0) = 0, \ u^{\prime}(1) = \int_{0}^{1}g(s)u^{\prime}(s)ds,
\end{equation*}
where $f$ and $g$ are continuous functions. By using the Krasnoselskii fixed-point theorem, some sufficient conditions are obtained for the existence and nonexistence of positive solution for the above problem.

In a recent paper \cite{Niu}, by using the well-known Avery-Henderson two fixed-
point theorem, B. W. Niu, J. P. Sun and Q. Y. Ren proved the existence of at least
two positive and decreasing solutions for the following third-order boundary value problem with integral boundary condition

\begin{equation*}
u^{\prime \prime \prime}(t)=f(t,u(t)), \ t \in [0,1],
\end{equation*}
\begin{equation*}
u^{\prime}(0) =u(1)= 0, \ u^{\prime \prime}(\eta)+ \int_{\alpha}^{\beta}u(s)ds=0,
\end{equation*}
where $1/2<\alpha\leq \beta \leq 1$, $\alpha+\beta\leq 4/3$ and $\eta\in (1/2,\alpha]$.

In \cite{Sun2}, Y. Sun, M. Zhao, and S. Li considered the following nonlinear third-order two-point boundary value problem
\begin{equation*}
u^{\prime \prime \prime}(t) + f(u(t)) = 0,\  t \in (0,1),
\end{equation*}
\begin{equation*}
u(0)=u^{\prime\prime}(0) = u^{\prime}(1) =0.
\end{equation*}
Under suitable assumptions on $f$ and by employing a fixed point theorem of cone expansion and
compression of functional type due to Avery, Anderson and Krueger, they established the existence
of at least one positive solution of the above boundary value problem.

Very recently, in 2016, S. Benaicha and F. Haddouchi \cite{Ben} investigated existence of positive solutions of the fourth-order integral boundary value problem
\begin{equation*}
u^{\prime \prime \prime \prime}(t) + f(u(t)) = 0,\  t \in (0,1),
\end{equation*}
\begin{equation*}
u^{\prime} (0) = u^{\prime} (1) =u^{\prime \prime}(0) = 0,\ u(0) = \int_{0}^{1}a(s)u(s)ds.
\end{equation*}
By using the Krasnoselskii's  fixed point theorem, results on existence of positive solutions
are presented.

And, in the same year, inspired greatly by \cite{Tor}, Ali Rezaiguia and Smail Kelaiaia \cite{Rez} investigated the following third-order three point boundary value problem
\begin{equation*}
u^{\prime \prime \prime}(t) +a(t) f(t,u(t)) = 0,\  t \in (0,1),
\end{equation*}
\begin{equation*}
u^{\prime}(0) = u^{\prime}(1) = \alpha u(\eta),\ u(0) = \beta u(\eta).
\end{equation*}
By using Krasnoselskii's fixed point theorem in cones, they generalized the work of Torres \cite{Tor}.

Inspired and motivated by the above recent works, we intend in this paper
to study the existence of at least one positive solution for \eqref{eq1} and \eqref{eq2}  if the nonlinearity $f$ is either superlinear or sublinear.

The rest of this paper is arranged as follows:\\
In Section 2, we present the necessary definitions
and we give some lemmas in order to prove our main results. In particular, we state some properties of the Green's function associated with BVP \eqref{eq1} and \eqref{eq2}. In Section 3, some sufficient conditions are established for the existence of positive solution to our BVP when $f$ is superlinear or sublinear. Finally, two examples are also included to illustrate the main results.\\

\section{Preliminaries}
At first, we consider the Banach space $C([0,1])$  equipped with the sup norm \[\|u\|=\\sup_{t\in[0, 1]}|u(t)|.\]
\begin{definition}
Let $E$ be a real Banach space. A nonempty, closed, convex set $
K\subset E$ is a cone if it satisfies the following two conditions:
\begin{itemize}
\item[(i)]
 $x\in K$, $\lambda \geq 0$ imply $\lambda x\in K$;
\item[(ii)]
$x\in K$, $-x\in K$ imply $x=0$.
\end{itemize}
\end{definition}

\begin{definition}
An operator $T:E\rightarrow E$ \ is completely continuous if it is continuous
and maps bounded sets into relatively compact sets.
\end{definition}
\begin{definition}
$ A $ function $ u(t)$ is called a positive solution of \eqref{eq1} and \eqref{eq2}
if $ u \in C ([0,1]) $ and $ u(t) > 0 $ for all $ t \in (0,1)$
\end{definition}

To prove our results, we need the following well-known fixed point theorem
of cone expansion and compression of norm type due to Krasnoselskii \cite{Kras}.

\begin{theorem}\label{T1}
Let $E$ be a Banach space, and let $K\subset E$, be a cone. Assume that $%
\Omega_{1}$ and $\Omega_{2}$ are open subsets of $E$ with $0\in \Omega _{1}$,
$\Omega _{1}\subset \Omega_{2}$ and let
\[
A:K\cap (\overline{
\Omega_{2}}\backslash \Omega_{1})\rightarrow K
\]
be a completely continuous operator such that
\begin{itemize}
\item[(a)]
$\left\Vert Au\right\Vert \leq \left\Vert u\right\Vert ,$ $u\in K\cap
\partial
\Omega _{1}$, and $\left\Vert Au\right\Vert \geq \left\Vert u\right\Vert ,$
$u\in K\cap \partial \Omega_{2}$; or
\item[(b)]
$\left\Vert Au\right\Vert \geq \left\Vert u\right\Vert ,$ $u\in K\cap
\partial
\Omega_{1}$, and \ $\left\Vert Au\right\Vert \leq \left\Vert u\right\Vert ,$
$u\in K\cap \partial \Omega_{2}.$
\end{itemize}
Then $A$ has a fixed point in $K\cap (\overline{\Omega _{2}}$ $\backslash $ $
\Omega_{1})$.
\end{theorem}

Consider the two-point boundary value problem

\begin{equation}\label{eq3}
u^{\prime \prime \prime}(t)+ h(t) = 0,\  t \in (0,1),
\end{equation}
\begin{equation}\label{eq4}
u(0) = u^{\prime}(0) = 0, \ u(1) = \int_{0}^{1}g(s)u(s)ds.
\end{equation}

\begin{lemma}\label{l1}
The problem \eqref{eq3}-\eqref{eq4} has a unique solution which can be expressed by
$$ u(t) = \int_{0}^{1} G(t,s) h(s) ds + \frac{t^{2}}{1-\mu} \int_{0}^{1} g(\tau)\left[ \int_{0}^{1}G(\tau,s)h(s)ds\right]d\tau ,$$
where $ G(t,s) : [0,1] \times [0,1]\rightarrow \mathbb{R}$ is the Green's function defined by
\begin{equation}\label{eq5}
G(t,s)=\frac{1}{2} \begin{cases}s(1-t)(2t-ts-s), & 0 \leq s \leq t \leq 1; \\
(1-s)^2 t^2, & 0 \leq t \leq s \leq 1,
\end{cases}
\end{equation}
and $ \mu = \int_{0}^{1} t^{2}g(t)dt$.
\end{lemma}
\begin{proof}
Rewriting \eqref{eq3} as $u^{\prime \prime \prime}(t)=-h(t)$ and integrating three times over the interval $ [0, t]$ for $ t \in [0,1] $, we obtain
$$ u^{\prime \prime}(t) = - \int_{0}^{t} h(s) ds + A, $$
$$ u^{\prime}(t) = - \int_{0}^{t} (t-s)h(s) ds + A t + B, $$
$$ u(t) = -\frac{1}{2} \int_{0}^{t}( t- s)^{2} h(s) ds +\frac{1}{2} A t^{2} + B t + C, $$
where $A, B, C \in \mathbb{R}$ are constants.
By \eqref{eq4}, we get $ B = C = 0$ and

$$ \frac{1}{2} A = \frac{1}{2} \int_{0}^{1} (1-s)^{2} h(s) ds + u(1). $$
Then
\begin{gather*}
\begin{aligned}
u(t)& =- \frac{1}{2}\int_0 ^t (t-s)^{2} h(s)ds +\frac{t^2}{2}\int_0 ^1(1-s)^{2}h(s)ds + t^{2} u(1) \\
&  =- \frac{1}{2}\int_0 ^t (t-s)^{2} h(s)ds +\frac{t^2}{2}\int_0 ^t(1-s)^{2}h(s)ds +\frac{t^2}{2}\int_t ^1(1-s)^{2}h(s)ds + t^{2} u(1) \\
&=\frac{1}{2}\int_0 ^t[ t^{2}(1-s)^{2} - (t-s )^{2}] h(s)ds +\frac{t^2}{2}\int_t ^1(1-s)^{2}h(s)ds + t^{2} u(1)\\
&=\frac{1}{2}\int_0 ^t s(1-t)(2t- ts -s) h(s)ds +\frac{t^2}{2}\int_t ^1(1-s)^{2}h(s)ds + t^{2} u(1).
\end{aligned}
\end{gather*}

 So
 \begin{equation} \label{eq6}
   u(t)= \int_0 ^1 G(t,s)h(s) ds + t^{2}u(1).
 \end{equation}
 Multiplying \eqref{eq6} by $g(t)$ and integrating the result over the interval $[0,1]$, we obtain
\begin{equation*}
 u(1) =\int_{0}^{1} g(\tau) u(\tau) d\tau = \int_{0}^{1} g(\tau) \Big(\int_0 ^1 G(\tau,s)h(s) ds \Big) d\tau + u(1) \int_{0}^{1} \tau^{2} g(\tau)d\tau,
  \end{equation*}
which implies
\begin{equation*}
 u(1) =\frac{1}{1-\mu} \int_{0}^{1} g(\tau) \left[ \int_{0}^{1} G(\tau,s) h(s)ds \right] d\tau.
\end{equation*}
Replacing this expression in \eqref{eq6}, we get
$$ u(t) = \int_{0}^{1} G(t,s) h(s)ds + \frac{t^2}{1-\mu} \int_{0}^{1} g(\tau) \left[ \int_{0}^{1} G(\tau,s) h(s)ds \right] d\tau.$$
\end{proof}

\begin{lemma}\label{l2} $ G(t,s)$ defined by \eqref{eq5} satisfies
\begin{itemize}
\item[(i)] $ G(t,s) \geq 0 $, for all \ $ t, s \in [0,1];$
\item[(ii)] $ \rho (t) s (1-s)^{2} \leq G(t,s) \leq s(1-s)^{2},$ \  for all \ $ t,s \in [0,1]$,
 where
\begin{equation}\label{eq7}
\rho(t)=\frac{1}{2}\min\{t^{2}, t(1-t)\}=\frac{1}{2}\begin{cases} t^{2},& t\leq \frac{1}{2};  \\
t(1-t), & t \geq \frac{1}{2}.
\end{cases}
\end{equation}
 \item[(iii)] Let $\theta\in ]0,\frac{1}{2}[ $ be fixed. Then\\
$\frac{\theta^{2}}{2}s(1-s)^{2} \leq G(t,s) \leq s(1-s)^{2},$  \ for all \  $(t,s) \in [\theta , 1- \theta] \times [0,1].$
 \end{itemize}
\end{lemma}
\begin{proof}
{\rm (i)} Since it is obvious for $ t \leq s $, we only need to prove the case $s \leq t $.\\
Assume that $ s \leq t $ , then
\begin{equation}\label{eq8}
\begin{split}
 G(t,s) &= \frac{1}{2} s(1-t) (2t - ts -s )\\
  & = \frac{1}{2} s(1-t) [(t-s) + t(1-s)]\\
   & \geq 0.
 \end{split}
 \end{equation}\\
{\rm (ii)} If $ s \leq t $, then from \eqref{eq5} we have
\begin{equation}\label{eq9}
\begin{split}
 G(t,s) &= \frac{1}{2} s(1-t) (2t - ts -s )\\
  & = \frac{1}{2} s(1-t) [(t-s) + t(1-s)]\\
   & \leq  \frac{1}{2} s(1-s) [(1-s) + (1-s)]\\
   & = s(1-s)^{2}.
 \end{split}
 \end{equation}\\
 On the other hand, we have
 \begin{equation}\label{eq10}
 \begin{split}
  G(t,s) &= \frac{1}{2} s(1-t)[(t-s)+t(1-s)]\\
   & \geq \frac{1}{2} s(1-t)t (1-s)\\
    & \geq  \frac{1}{2} t(1-t) s(1-s)^{2}.\\
  \end{split}
  \end{equation}\\
If $ t \leq s $, from \eqref{eq5}, we have
   \begin{equation}\label{eq11}
   \begin{split}
    G(t,s) &= \frac{1}{2} t^{2} (1-s)^{2}  \\
     & \leq  s^{2} (1-s)^{2} \\
      & \leq s(1-s)^{2},\\
    \end{split}
    \end{equation}
    and,
   \begin{equation}\label{eq12}
      \begin{split}
       G(t,s) &= \frac{1}{2} t^{2} (1-s)^{2} \\
        & \geq  \frac{1}{2} t^{2} s (1-s)^{2}. \\
       \end{split}
       \end{equation}\\
From \eqref{eq9}, \eqref{eq10}, \eqref{eq11} and \eqref{eq12}, we have
\[ \rho (t) s (1-s)^{2} \leq G(t,s) \leq s(1-s)^{2},\ \text{for all} \ (t,s) \in [0,1]\times[0,1].\]
{\rm (iii)} It follows immediately from {\rm (ii)}.
\end{proof}

\begin{lemma}\label{l4}
Let $ h(t) \in C ([0,1],[0,\infty )) $ and $ \theta \in ]0, \frac{1}{2}[$. The unique solution  of \eqref{eq3}-\eqref{eq4} is a nonnegative and satisfies
\[ \min_{t\in [\theta, 1-\theta]} u(t) \geq \frac{\theta^2}{2}\Big(\frac{1-\mu +\beta}{1-\mu +\alpha}\Big)\|u\|,\]
where $ \mu = \int_{0}^{1} t^2 g(t)dt$,\ $\beta= \theta^2 \int_{\theta}^{1-\theta}g(t) dt$,\ $\alpha = \int_{0}^{1}g(t)dt$.
\end{lemma}

\begin{proof}
From Lemma \ref{l1} and Lemma \ref{l2}, $ u(t) $ is nonnegative for $ t \in [0,1] $, and we get
\begin{equation}\label{eq13}
\begin{split}
u(t) &= \int_{0}^{1} G(t,s) h(s)ds + \frac{t^2}{1-\mu} \int_{0}^{1}g(\tau)\left[ \int_{0}^{1} G(\tau,s) h(s) ds \right] d\tau \\
&\leq \int_{0}^{1} s(1-s)^2 \left(1+\frac{\alpha}{1-\mu}\right)h(s)ds\\
&= \frac{1-\mu +\alpha}{1-\mu}\int_{0}^{1} s(1-s)^2 h(s)ds.
\end{split}
\end{equation}
Then
\begin{equation}\label{eq14}
\|u\| \leq \frac{1-\mu +\alpha}{1-\mu} \int_{0}^{1} s (1-s)^2 h(s)ds .
\end{equation}

For $ t\in [\theta, 1-\theta] $, we have
\begin{equation} \label{eq15}
\begin{split}
u(t)&=\int_{0}^{1}G(t,s) h(s) ds + \frac{t^2}{1-\mu} \int_{0}^{1}g(\tau) \left[\int_{0}^{1}G(\tau,s) h(s)ds\right]d\tau  \\
&\geq\frac{\theta^2}{2} \int_{0}^{1} s (1-s)^2 h(s)ds + \frac{\theta^2}{1-\mu} \int_{\theta}^{1-\theta} g(\tau)\left[\int_{0}^{1}\frac{\theta^{2}}{2}s(1-s)^2 h(s)ds\right]d\tau   \\
&=\frac{\theta^2}{2}\left[ \int_{0}^{1} s (1-s)^2 h(s)ds + \frac{\beta}{1-\mu} \int_{0}^{1}  s(1-s)^2 h(s)ds\right]\\
&=\frac{\theta^2}{2} \left[ \int_{0}^{1} s(1-s)^2 \left( 1+ \frac{\beta}{1-\mu}\right)h(s)ds\right] \\
&=\frac{\theta^2}{2}\left(\frac{1-\mu +\beta}{1-\mu}\right)\int_{0}^{1} s(1-s)^2 h(s)ds.
\end{split}
\end{equation}
From \eqref{eq14} and \eqref{eq15}, we obtain $$\min_{t\in [\theta, 1-\theta]} u(t) \geq  \frac{\theta^2}{2}\left(\frac{1-\mu +\beta}{1-\mu + \alpha }\right) \|u\|.$$
\end{proof}
Let $ \theta \in ]0, \frac{1}{2}[ $, we define the cone $$K= \left\lbrace u \in C([0,1],\ \mathbb{R}),\ u \geq 0 :  \min_{t\in [\theta ,1-\theta]} u(t) \geq  \frac{\theta^2}{2}\left(\frac{1-\mu +\beta}{1-\mu + \alpha }\right) \|u\| \right\rbrace,$$
and the operator $ A :K \rightarrow C[0,1] $ by
\begin{equation}\label{eq16}
Au(t)= \int_0 ^1 G(t,s)f(u(s))ds +\frac{t^2}{1-\mu}\int_0 ^1 g(\tau)\left[ \int_0 ^1 G(\tau,s)f(u(s))ds \right]d\tau .
\end{equation}
\begin{remark}
By Lemma \ref{l1}, problem \eqref{eq1}-\eqref{eq2} has a positive solution $ u(t)$ if and only if $ u $ is a fixed point of $ A $.
\end{remark}
\begin{lemma}
The operator $ A $ defined in \eqref{eq16} is completely continuous and satisfies $ AK \subset K$ .
\end{lemma}
\begin{proof}
From Lemma \ref{l4}, we obtain $  AK \subset K $. By an application of Arzela-Ascoli theorem $ A$ is completely continuous.
\end{proof}

For convenience, we introduce the following notations
$$ f_{0} = \lim_{u\rightarrow 0^{+}} \frac{f(u)}{u} ,\hspace{2mm} f_{\infty} = \lim_{u\rightarrow +\infty} \frac{f(u)}{u} $$

\section{Existence of positive solutions}
In this section , we will state and prove our main results.
\begin{theorem}\label{T2}
Assume that $ f_{0} = 0 $ and $ f_{\infty} = \infty $ (superlinear). Then the  problem \eqref{eq1} and \eqref{eq2} has at least one positive solution.
\end{theorem}
\begin{proof}
Since $ f_{0} = 0 $, there exists  $ \rho_1 > 0 $ such that $ f(u) \leq \epsilon u $, for $0 <u \leq \rho_1 $, where  $ \epsilon > 0 $ satisfies
$$\frac{\epsilon (1 - \mu + \alpha)}{1 - \mu } \leq 1 .$$
Thus, if we let $$ \Omega_1 = \{u \in C[0,1]:\|u\| < \rho_1 \} ,$$
then, for $ u \in K \cap \partial \Omega_1$, we have
\begin{equation}\label{eq17}
\begin{split}
Au(t)&= \int_0 ^1 G(t,s)f(u(s))ds +\frac{t^2}{1-\mu}\int_0 ^1 g(\tau)\left[ \int_0 ^1 G(\tau,s)f(u(s))ds \right]d\tau \\
& \leq \int_0 ^1 G(t,s)f(u(s))ds +\frac{1}{1-\mu}\int_0 ^1 g(\tau)\left[ \int_0 ^1 G(\tau,s)f(u(s))ds \right]d\tau \\
&\leq \int_0 ^1 s (1-s)^2 \epsilon u(s)ds + \frac{\alpha}{1-\mu}\int_0 ^1 s(1-s)^2 \epsilon u(s)ds\\
&\leq \frac{\epsilon (1-\mu + \alpha)}{1-\mu}\|u\| \int_0 ^1 s (1-s)^2 ds\\
&\leq \|u\|,
\end{split}
\end{equation}
which yields $$ \|Au \|\leq \|u\|,\ u \in K \cap \partial \Omega_1 .$$
On the other hand, since $ f_\infty =\infty $, there exists $ \hat{\rho}_2  > 0$ such that $ f(u)\geq \delta u$, for $ u >  \hat{\rho}_2 $, where $ \delta >0 $  is chosen so that
\begin{equation*}
\delta \frac{\theta^4}{24}\frac{(1-\mu + \beta)^2}{(1-\mu)(1-\mu + \alpha)}(1-2\theta)\left(\frac{1}{2} +\theta -\theta^2\right)\geq 1.
\end{equation*}
Let $ \rho_2 = \max \left\lbrace 2 \rho_1 ,\frac{2 \hat{\rho}_2 (1-\mu + \alpha)}{\theta^2 (1-\mu + \beta)}\right\rbrace $ and $ \Omega_2 = \left\lbrace u \in C[0,1]: \|u \| < \rho_2 \right\rbrace $. Then $ u \in K \cap \partial \Omega_2 $ implies that
\[ \min _{t\in [\theta , 1- \theta]}u(t) \geq \frac{\theta ^2}{2}\left( \frac{1-\mu +\beta}{1-\mu + \alpha}\right) \|u\|= \frac{\theta ^2}{2}\left( \frac{1-\mu +\beta}{1-\mu + \alpha}\right)\rho_2 \geq \hat{\rho}_2.\]
So, by \eqref{eq16} and for $ t \in [\theta, 1-\theta]$, we get
\begin{equation}\label{eq18}
\begin{split}
Au(t)&= \int_{0}^{1}G(t,s)f(u(s))ds + \frac{t^2}{1-\mu}\int_{0}^{1}g(\tau)\left[\int_{0}^{1}G(\tau,s)f(u(s))ds \right] d\tau\\
&\geq \frac{\theta^2}{2}\int_{\theta}^{1-\theta} s(1-s)^2 \delta u(s)ds+\frac{\theta^2}{1-\mu}\int_{\theta}^{1-\theta}g(\tau)\left[\int_{\theta}^{1-\theta}\frac{\theta^2}{2}s(1-s)^2 \delta u(s) ds \right]d\tau\\
&= \frac{\theta^2}{2}\int_{\theta}^{1-\theta} s(1-s)^2 \delta u(s)ds + \frac{\beta}{1-\mu} \int_{\theta}^{1-\theta}\frac{\theta^2}{2}s(1-s)^2 \delta u(s) ds \\
&= \frac{\theta^2}{2}\delta \left(\frac{1-\mu +\beta }{1-\mu } \right)\int_{\theta}^{1-\theta}s(1-s)^2 u(s)ds \\
&\geq   \frac{\theta^2}{2}\delta \left(\frac{1-\mu +\beta }{1-\mu } \right)\min_{t\in [\theta ,1-\theta]} u(t)\int_{\theta}^{1-\theta}s(1-s)^2 ds \\
&\geq   \frac{\theta^2}{2}\delta \left(\frac{1-\mu +\beta }{1-\mu } \right)\frac{\theta^2}{2}\left(\frac{1-\mu +\beta}{1-\mu +\alpha}\right) \frac{1}{6}(1-2\theta)\left(\frac{1}{2}+\theta - \theta^2\right) \|u\| \\
&= \delta \frac{\theta^4 (1-\mu + \beta)^2}{24(1-\mu)(1-\mu +\alpha)} (1-2\theta)\left(\frac{1}{2}+\theta -\theta^2\right) \|u\| \\
&\geq \|u\| .
\end{split}
\end{equation}
Therefore, $ \|Au\| \geq \|u\|,\ u\in K\cap \partial \Omega_2 $.\\
By Theorem \ref{T1}, the operator  $ A $ has a fixed point in $ K\cap (\overline{\Omega_{2}}  \setminus \Omega_1) $ such that $ \rho_1 \leq \|u\| \leq \rho_{2} $, which is a solution of the problem \eqref{eq1} and \eqref{eq2}.
\end{proof}

\begin{theorem} \label{T3}
Assume that $ f_0 = \infty  $ and $ f_{\infty} =0$ (sublinear). Then the  problem \eqref{eq1} and \eqref{eq2} has at least one positive solution .
\end{theorem}
\begin{proof}
Since  $ f_{0} = \infty $, there exists $ \rho_1 >0 $ such that $ f(u) \geq \lambda u $, for $ 0 < u\leq \rho_1$,\\ where $ \lambda $ is chosen so that
\begin{equation*}
\lambda \frac{\theta^4 (1-\mu + \beta)^2}{24(1-\mu)(1-\mu +\alpha)} (1-2\theta)\left(\frac{1}{2}+\theta -\theta^2\right)\geq 1.
\end{equation*}
Thus, for $ u \in K \cap \partial \Omega_1 $ with
\begin{equation*}
\Omega_1 = \{u \in C[0,1]:\|u\| < \rho_1\} ,
\end{equation*}
and by \eqref{eq18}, we get
\begin{equation}\label{eq19}
\begin{split}
Au(t)&= \int_{0}^{1}G(t,s)f(u(s))ds  + \frac{t^2}{1-\mu}\int_{0}^{1}g(\tau)\left[\int_{0}^{1}G(\tau,s)f(u(s))ds \right] d\tau\\
&\geq \frac{\theta^{2}}{2}\int_{\theta}^{1-\theta} s(1-s)^{2} \lambda u(s)ds  + \frac{\theta^2}{1-\mu}\int_{\theta}^{1-\theta}g(\tau)\left[\int_{\theta}^{1-\theta} \frac{\theta^{2}}{2} s (1-s)^{2} \lambda u(s)ds \right] d\tau \\
&\geq \lambda \frac{\theta^4 (1-\mu + \beta)^2}{24(1-\mu)(1-\mu +\alpha)} (1-2\theta)\left(\frac{1}{2}+\theta -\theta^2\right) \|u\| \\
&\geq \|u\|.
\end{split}
\end{equation}

Then, $ Au(t) \geq \| u \| $ for $ t \in [\theta , 1- \theta] $, which implies that
\begin{equation*}
\|Au\| \geq \|u\|,\hspace{2mm}u \in K\cap \partial \Omega_1 .
\end{equation*}
Now, we construct the set  $ \Omega_2 $. We have two  possible cases :  \\
Case 1. If $ f $ is bounded. Then there exists $ L > 0 $ such that $ f(u) \leq L $.\\
Set $ \Omega_2 = \{u \in C [0,1]: \|u\| < \rho_2 \} $, where $ \rho_2 = \max \left\lbrace 2\rho_1 , \frac{L(1-\mu +\alpha)}{1-\mu}\right\rbrace $. \\ If $ u \in K \cap \partial \Omega_2 $ , similar to the estimates of \eqref{eq17}, we get
\begin{equation*}
\begin{split}
Au(t) &\leq \frac{L(1-\mu + \alpha)}{1-\mu}\int_{0}^{1}s(1-s)^2 ds \\
&\leq \frac{L(1-\mu + \alpha)}{1-\mu} \leq \rho_2 = \|u\|,
\end{split}
\end{equation*}
and consequently,
$$ \|Au \| \leq \|u\|,\ u \in K \cap \partial \Omega_2 .$$
Case 2. Suppose that $ f $ is unbounded. Because $ f_{\infty} =0 $, so there exists $ \hat{\rho}_2 > 0\hspace{2mm} (\hat{\rho}_2 >\rho_1) $ such that $ f(u) \leq \eta u $ for $ u > \hat{\rho}_2 $, where  $ \eta > 0 $ satisfies \\  $$ \frac{\eta (1-\mu +\alpha)}{1-\mu}\leq 1 . $$
But, from condition \rm{(H1)}, there exists $ \sigma > 0 $ such that $ f(u) \leq \eta \sigma $, with $ 0 \leq u \leq \hat{\rho}_2 $. \\
Set  $ \Omega_2 = \{ u \in C[0,1]: \|u\| < \rho_2 \} $, where  $ \rho_2 = \max \{ \sigma , \hat{\rho}_2\} $.\\
If $ u \in K \cap \partial \Omega_2 $, then we have $ f(u) \leq \eta \rho_{2} $. Similar to the estimates  \eqref{eq17}, we obtain
\begin{equation*}
\begin{split}
Au(t)&\leq \frac{\eta \rho_2 (1-\mu + \alpha)}{1-\mu}\int_{0}^{1}s(1-s)^2 ds\\
&\leq \frac{\eta \rho_2 (1-\mu + \alpha)}{1-\mu} \leq \rho_2 = \|u\|.
\end{split}
\end{equation*}
This implies that $\|Au \| \leq \|u\|$, for  $ u \in K \cap \partial \Omega_2$.\\

Therefore, it follows from Theorem \ref{T1} that $ A $ has a fixed point in $ u \in K \cap( \overline{\Omega_{2}} \setminus \Omega_{1}) $, which is a positive solution of problem \eqref{eq1} and \eqref{eq2} .
\end{proof}

\section{Examples}

\begin{example}Consider the boundary value problem
\begin{equation}\label{eq20}
u^{\prime \prime \prime }(t) + u^{2} e^{u} = 0,\ t \in (0,1),
\end{equation}
\begin{equation}\label{eq21}
u(0) = u^{\prime}(0)= 0,\ u(1) = \int_{0}^{1}s^{4} u(s)ds,
\end{equation}
where $ f(u) = u^{2} e^{u} \in C( [0, \infty), [0,\infty ))$ and $ g(t)= t^{4}\geq 0$,
$ \mu = \int_{0}^{1} t^{2}  g(t)dt = \int_{0}^{1} t^{6} dt = \frac{1}{7} < 1.$ \\
We have
\begin{align*}
\lim_{u\rightarrow 0^{+}} \frac{f(u)}{u}& = \lim_{u\rightarrow 0^{+}} \frac{u^{2} e^{u} }{u} = 0,\\
\lim_{u\rightarrow +\infty } \frac{f(u)}{u} &= \lim_{u\rightarrow +\infty } \frac{u^{2} e^{u} }{u} = +\infty.
\end{align*}
Then, by Theorem \ref{T2}, the problem \eqref{eq20}-\eqref{eq21} has at least one positive solution .
\end{example}

\begin{example}Consider the boundary value problem
\begin{equation}\label{eq22}
u^{\prime \prime \prime }(t) + \sqrt{u} + \ln (1+u) = 0,\ t \in (0,1),
\end{equation}
\begin{equation}\label{eq23}
u(0) = u^{\prime}(0) = 0,\ u(1) = \int_{0}^{1} s^{6} u(s)ds,
\end{equation}
where $ g(t) = t^{6}$ and $ f(u) = \sqrt{u} + \ln (1+u) \in C( [0, \infty), [0,\infty ))$,
so $ \mu = \int_{0}^{1} s^{8} ds = \frac{1}{9} < 1$. We have
\begin{align*}
\lim_{u\rightarrow 0^{+}} \frac{f(u)}{u} &= \lim_{u\rightarrow 0^{+}} \frac{\sqrt{u} + \ln (1+u)}{u} = +\infty,\\
\lim_{u\rightarrow +\infty } \frac{f(u)}{u} &= \lim_{u\rightarrow +\infty } \frac{\sqrt{u} + \ln (1+u)}{u} = 0.
\end{align*}
From Theorem \ref{T3}, the problem \eqref{eq22}-\eqref{eq23} has at least one positive solution .
\end{example}

\end{document}